\documentclass[a4paper,12pt]{article}

\usepackage{fullpage}

\usepackage{libertine}
\usepackage{libertinust1math}
\usepackage[T1]{fontenc}

\usepackage{amssymb}
\usepackage{amsfonts}
\usepackage{amsmath}
\usepackage{amsthm}

\RequirePackage[l2tabu, orthodox]{nag}
\usepackage[all,warning]{onlyamsmath}

\usepackage{xcolor}
\usepackage[colorlinks=true,citecolor=cyan, linkcolor=violet]{hyperref}
\usepackage{cleveref}
\usepackage{mathtools}
\usepackage{xspace}

\usepackage{algorithmicx}
\usepackage{algpseudocode}

\usepackage{enumerate}
\usepackage{MnSymbol}

\theoremstyle{plain}
\newtheorem{theorem}{Theorem}[section]
\newtheorem{lemma}[theorem]{Lemma}

\newtheorem{proposition}[theorem]{Proposition}

\theoremstyle{definition}

\newtheorem{example}[theorem]{Example}

\newcommand{\set}[1]{\left\{#1\right\}}
\newcommand{\st}{\ :\ }

\newcommand{\nat}{\mathbb{N}}
\newcommand{\intg}{\mathbb{Z}}
\newcommand{\rel}{\mathbb{R}}
\newcommand{\rat}{\mathbb{Q}}

\newcommand{\pam}{{\sc pam}\xspace}

\newcommand{\ie}{{\em i.e.}\xspace}
\newcommand{\eg}{{\em e.g.}\xspace}
\newcommand{\etc}{{\em etc.}\xspace}

\newcommand\defeq{\mathrel{\overset{\makebox[0pt]{\mbox{\normalfont\tiny\sffamily def}}}{=}}}

\renewcommand{\mod}[1]{\qquad (\mathrm{mod}\ #1)}

\newcommand{\unit}{U}
\newcommand\restrict[1]{\raisebox{-.5ex}{$|$}_{#1}}
\newcommand{\im}{\mathrm{Im}}

\newcommand\orbit{\mathcal O}
\newcommand\algo{\mathbb A}

\newcommand\mfbox{\fcolorbox{gray}{white}}

\usepackage{thmtools}

\title{Reachability in Injective Piecewise~Affine~Maps}
\author{Faraz Ghahremani\footnote{Sharif University of Technology, Iran. \href{mailto:faraz.ghahremani@yahoo.com}{faraz.ghahremani@yahoo.com}.}
  \and Edon Kelmendi\footnote{Queen Mary University of London, UK. \href{mailto:e.kelmendi@qmul.ac.uk}{e.kelmendi@qmul.ac.uk}}
  \and Jo\"el Ouaknine\footnote{Max Planck Insitute for Software Systems, Saarland University Campus, Germany. Also affiliated with Keble College, Oxford as \href{http://emmy.network/}{emmy.network} Fellow, and supported by DFG grant 389792660 as part of TRR 248 (see \href{https://perspicuous-computing.science}{https://perspicuous-computing.science}). \href{mailto:joel@mpi-sws.org}{joel@mpi-sws.org}.}} 
\date{}
\begin{document}
\maketitle
\begin{abstract}
  One of the most basic, longstanding open problems in the theory of dynamical systems is whether reachability is decidable for one-dimensional piecewise affine maps with two intervals. In this paper we prove that for injective maps, it is decidable.

  We also study various related problems, in each case either establishing decidability, or showing that they are closely connected to Diophantine properties of certain transcendental numbers, analogous to the positivity problem for linear recurrence sequences. Lastly, we consider topological properties of orbits of one-dimensional piecewise affine maps, not necessarily with two intervals, and negatively answer a question of Bournez, Kurganskyy, and Potapov, about the set of orbits in expanding maps. 
\end{abstract}
\section{Introduction}
A \textbf{piecewise affine map} (abbreviated \pam) is a map from the half-open unit interval to itself
\begin{align*}
  f\st [0,1)\to [0,1),
\end{align*}
with the property that there exist half-open intervals $I_1,\ldots, I_\ell$ partitioning $[0,1)$ such that the restriction of $f$ to any of these intervals $I_k$ is affine:
\begin{align*}
  f\restrict{I_k}(x)\equiv a_kx+b_k,
\end{align*}
where the constants $a_k,b_k$ are rational.

Here is an example of a \pam with two intervals:
\begin{align*}
  f(x)=
  \begin{cases}
    \frac 2 3 x + \frac 2 3 & \text{ if }0\le x < \frac 1 2,\\
    \frac 4 3 x - \frac 2 3 & \text{ if }\frac 1 2 \le x < 1.
  \end{cases}  
\end{align*}
The basic decision problem about orbits of these maps is the \textbf{reachability problem}: Given a \pam~$f$, an initial point $x_0$, and a target point $t$ in the unit interval, does there exist some natural number $n\in\nat$ such that
\begin{align*}
  f^n(x_0)\defeq \underbrace{f(f(\cdots f}_{n\text{ times }}(x_0)\cdots )=t,
\end{align*}
in other words, can we reach $t$ from $x_0$ by repeatedly applying~$f$? It is a longstanding open question whether there is a procedure to decide reachability for piecewise affine maps, even when there are only two intervals $I_1,I_2$.

Although such maps have long been studied in the dynamical systems community, to the best of our knowledge the first explicit mention of the reachability decision problem appeared in the 1994 paper of Koiran, Cosnard, and Garzon~\cite{koiran94_comput_with_low_dimen_dynam_system}. In that paper the authors prove that \emph{two-dimensional} piecewise affine maps (where instead of intervals, a bounded region of the plane is partitioned into polytopes) have an undecidable reachability problem (Theorem 3.1); however they leave the one-dimensional problem open. 

Why is this simple problem considered interesting? There are two main reasons. One, it is arguably the simplest type of dynamical system that can have immensely complicated behaviour, and two, the reachability problem for \pam is related to numerous problems in mathematics and computer science.

We mention two families of piecewise affine maps that have been studied in some depth and possess a rich theory with connections to many parts of mathematics: interval exchange transformations and $\beta$-expansions. The former are maps that permute the intervals, and are related to Abelian differentials, continued fraction expansions, polygonal billiards \etc, see~\cite{viana06_ergod_theor_inter_exchan_maps}, and have had recent resurgence in interest due to a number of breakthroughs. The latter are expansions of numbers in nonintegral bases, and still shrouded in mystery (except from the metrical point of view),  see \cite[Chapter~9]{bugeaud12}, \cite[Section~5]{bournez18_reach_probl_one_dimen_piecew_affin_maps}, \cite[Section~7]{boker15_target_discoun_sum_probl}.

From a computer-science perspective, \pam are related to recurrent neural networks and cellular automata \cite{koiran94_comput_with_low_dimen_dynam_system}. The target discounted-sum problem \cite{boker15_target_discoun_sum_probl} is reducible to reachability for one dimensional \pam, and so are a plethora of open problems about discounted automata and games, inclusion in a generalised Cantor set \etc Furthermore, \pam play a central role in the study of hybrid systems, and the literature is replete with reductions from the reachability problem for \pam to establish hardness or openness of various decision problems. See \eg \cite{asarin02_widen_bound_decid_undec_hybrid_system, asarin12}.  

Finally, another important motivation for studying such maps is the following. \pam reachability is equivalent to the halting problem for a simple family of programs, namely loops of the following kind:
\begin{center}
\mfbox{  
\begin{minipage}{0.25\textwidth}
\begin{algorithmic}[H]
    \State {$x\gets x_0$}
    \While {$x\ne t$}
      \If {$x<c$}
        \State $x\gets a_1 x+b_1$
      \Else
        \State $x\gets a_2 x +b_2$
      \EndIf
    \EndWhile
  \end{algorithmic}
\end{minipage}
}
\end{center}
Algorithms able to decide the halting problem for these simple loops have potential utility in the wider framework of software verification. 

What is known about the reachability problem for piecewise affine maps? Here are two results directly related to the present paper: First, for purely affine maps, the corresponding reachability question is known as the \emph{orbit problem} and has been shown in the 1980s by Kannan and Lipton to be decidable in all dimensions~\cite{kannan86_polyn_time_algor_orbit_probl}. Second, as mentioned above, for two-dimensional \pam the reachability problem is undecidable~\cite[Theorem~3.1]{koiran94_comput_with_low_dimen_dynam_system}, hence the present restriction to one-dimensional maps.
Moreover, several problems related to reachability have been considered by Koiran and co-authors, see for example \cite{koiran2001topological, blondel2001stability, blondel2001deciding}. 

More recently, Bournez, Kurganskyy, and Potapov \cite{bournez18_reach_probl_one_dimen_piecew_affin_maps,kurganskyy2016reachability}, using $p$-adic valuations, proved (among other things) that complete \pam with two intervals have a decidable reachability problem. A \pam with two intervals is said to be \textbf{complete} if the image of each affine component is the unit interval, that is:
\begin{align*}
  \im(f\restrict{I_1})=\im(f\restrict{I_2})=[0,1). 
\end{align*}
There have also been various important results on the dynamics of \pam from researchers in the ergodic-theory, number-theory, and dynamical-systems communities, but none---to the best of our knowledge---enabling us to design algorithms for the decision problems that concern us. 

\subsection{Contributions}
\noindent The main contribution of this paper is the following:
\begin{theorem}
  \label{th:injective}
  The reachability problem is decidable for injective piecewise affine maps with two intervals. 
\end{theorem}
It is interesting to compare and contrast our result to that of Bournez, Kurganskyy, and Potapov~\cite[Corollary~10]{bournez18_reach_probl_one_dimen_piecew_affin_maps}: their theorem solves the case of complete \pam, which are the extreme opposite of injective \pam that we consider (for every point $y$ in the image, there are two distinct points that are mapped to $y$). One might hope that the accumulation of techniques for handling these two extreme families can lead to a better understanding of the general problem.

The proof of \Cref{th:injective} makes use of two main tools. We begin by a number of observations that reduce the general problem to reachability in two families of \pam: bijections, and certain injective maps that have positive slopes and are not surjective. For the former, we utilise a topological conjugacy idea of Boshernitzan~\cite{boshernitzan93_dense_orbit_ration} to reduce the reachability problem to one in a much simpler dynamical system. For the latter, we crucially apply the recent detailed analysis of orbits of certain \pam due to Laurent and Nogueira in~\cite{LG21, LGarxiv}, which uses Hecke-Mahler series. 

We then present a couple additional results in~\Cref{sec:related problems}; these include natural extensions, such as point-to-interval and interval-to-interval reachability, as well as deciding whether the orbit of a given point $x_0\in [0,1)$ is periodic, \ie are there distinct positive integers $n,m$ such that 
\begin{align*}
  f^n(x_0)=f^m(x_0)?
\end{align*}

In~\Cref{sec:related problems} we also prove, perhaps surprisingly, that even for injective maps with two intervals, there are some natural problems that are rather difficult, due to their connection with Diophantine approximation. More precisely, we consider the question of whether there exists some $n\in\nat$ such that
\begin{align*}
  nf^n(x_0)<c,
\end{align*}
where $c\in\rat$ is given as input. A hypothetical procedure to decide this problem could also be used to compute the Lagrange constants of certain transcendental numbers. An analogous correspondence holds for the positivity problem for linear recurrence sequences \cite[Section~5]{ouaknine13_posit_probl_low_order_linear_recur_sequen}, however for a different set of transcendental numbers. 

In the last part of this paper, \Cref{sec:dichotomy}, we study topological properties of orbits. This is inspired by a question of Bournez \emph{et al.}\ (see Hypothesis~1 in \cite{bournez18_reach_probl_one_dimen_piecew_affin_maps} and in~\cite{kurganskyy2016reachability}). They ask whether all expanding \pam (\ie\ maps comprising affine components all of whose slopes are $>1$) have orbits that are either periodic, or dense in the whole unit interval. We answer this question negatively by exhibiting a counterexample, and also prove a weaker statement (which can likely be strengthened): namely that expanding \pam have orbits that are either periodic or have infinitely many accumulation points. 

\section{Reachability for Injective Maps}
\label{sec:reachability}
In this section we prove our main result, \Cref{th:injective}: the reachability problem is decidable for injective \pam with two intervals. To this end, we give a couple of definitions and establish some properties of general piecewise affine maps. 

It is sometimes convenient to depict \pam pictorially as:
\begin{center}
\mfbox{  
  \begin{minipage}{0.4\textwidth}
    \includegraphics[width=\textwidth]{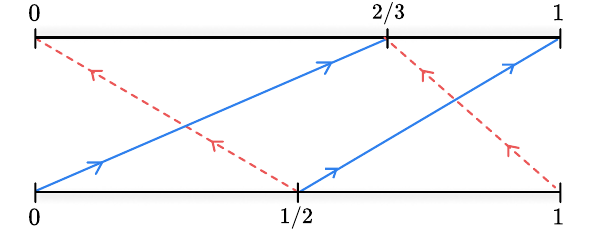}
  \end{minipage}
}
\end{center}
The picture has the following semantics: the two intervals are $[0,1/2)$ and $[1/2,1)$; the first one is mapped to $[2/3,1)$, and the second one is mapped to $[0,2/3)$. More precisely we have $f(0)=2/3$, $\lim_{x\to (1/2)^{-}}f(x)=1$, and for the second interval, $f(1/2)=0$, $\lim_{x\to 1^{-}}f(x)=2/3$. Note that this information is sufficient to specify the affine functions; the \pam depicted in the picture above is in fact exactly that given as an example in the introduction.

Denote the half-open unit interval by $\unit:=[0,1)$. Let $f$ be a \pam and $x_0\in\unit$. We denote the \textbf{orbit} of $x_0$ under $f$ by
\begin{align*}
  \orbit(f,x_0)\defeq \set{f^n(x_0)\st n\in\nat}. 
\end{align*}
There is nothing special about the unit interval $\unit$ in the definition of piecewise affine maps. Indeed, when given a \pam $f$ mapping some general interval $[a_1,a_2)$ to itself, we can reduce its reachability problem to that of a \pam from $\unit$ to $\unit$ (and \emph{vice-versa}). To see this, consider the bijection $h$ from $[a_1,a_2)$ to $\unit$ defined as:
\begin{align*}
  h(x)\defeq \frac {x-a_1} {a_2-a_1},
\end{align*} and define the function $g\st \unit\to\unit$, as 
\begin{align*}
  g\defeq h^{-1}\circ f\circ h. 
\end{align*}
Clearly $g^n=h^{-1}\circ f^n\circ h$, so $t\in [a_1,a_2)$ is reached from $x_0\in [a_1,a_2)$ by applying $f$, if and only if the same is true for $h(t)$ and $h(x_0)$ by applying $g$.

Now we identify a few cases of \pam where reachability is easy. For this we need the following definition.

Let $f$ be a \pam with intervals $I_1,\ldots, I_\ell$. The \textbf{interval reachability graph} of $f$ is a directed graph $G_f=(V,E)$, where the vertices are
  \begin{align*}
    V\defeq \set{0,1,\ldots,\ell},
  \end{align*}
  and there is an edge from vertex $k$ to vertex $j$ if and only if
  \begin{align*}
    \im(f\restrict{I_k})\cap I_j\ne\emptyset. 
  \end{align*}
  In other words, we put an edge from $k$ to $j$, if there is a point in interval $I_k$ that is mapped, via $f$, to a point in the interval $I_j$. When the interval reachability graph is particularly simple, we can decide reachability:
  \begin{lemma}
    \label{lem:easy}
    Let $f$ be a \pam and $G_f$ its interval reachability graph. Suppose moreover that $G_f$ has at least one of the following properties:
    \begin{enumerate}
    \item the only loops in $G_f$ are self-loops, 
    \item every vertex in $G_f$ has a unique outgoing edge. 
    \end{enumerate}
    Then there exists a procedure to decide reachability for $f$. 
  \end{lemma}
  \begin{proof}
    Let $g(x):=ax+b$ be an affine map, where $a,b\in\rat$. Note that if $a\ne 1$, then by a geometric-series argument, for all $n\in\nat$,
    \begin{align*}
      g^n(x)=a^nx+\frac{a^n-1}{a-1}b. 
    \end{align*}
    As a consequence, given $x_0,t\in\rel$ and an interval $I\subset\rel$, we can decide the following two questions by looking at the prime decompositions of the relevant numbers:
    \begin{enumerate}[(i)]
      \item Does there exist $n\in\nat$ such that $g^n(x_0)=t$?
      \item Does there exist $n\in\nat$ such that $g^n(x_0)\in I$?
    \end{enumerate}
    
    Let now $f$ be a \pam, $x_0,t\in \unit$ the initial and target points respectively, and $G_f$ its interval reachability graph. Suppose that the only loops in $G_f$ are self-loops. This means that once the trajectory of $x_0$ leaves some interval it will never go back to it again. Let $I_1,\ldots, I_\ell$ be the intervals of $f$, and $f_1,\ldots, f_\ell$ the corresponding affine maps. Suppose that $x_0\in I_k$, for some $k$. Using (ii), for the map $f_k$, we can decide whether the trajectory of $x_0$ always stays in $I_k$ or whether it leaves and reaches another interval~$I_{k'}$. If it stays forever in $I_k$, we decide reachability with (i), if it reaches $I_{k'}$, then we repeat this process, and hence identify an interval, say $I_j$, that contains the tail of the trajectory, \ie all but finitely many members of $\orbit(f,x_0)$ belong to $I_j$. Now to decide whether $t$ is reachable it suffices to check if it appears in the trajectory before $I_j$ is reached and if not, to use (i) for the affine map~$f_j$.

    Suppose now that every vertex in $G_f$ has a unique outgoing edge. This implies that there exists a loop of length $p$ in $G_f$ of the form:
    \begin{align*}
      k=k_0\to k_1 \to \cdots \to k_{p-1} \to k,
    \end{align*}
    and that 
    \begin{align*}
      f^p(I_{k_i})\subset I_{k_i},
    \end{align*}
    for all $i\in\set{0,1,\ldots, p-1}$. The function $f^p$ is itself a \pam, and since from every interval of $f$ we can go to a unique successor interval, we have that the intervals of $f^p$ coincide with those of $f$. In particular $I_{k_0}, I_{k_1},\ldots, I_{k_{p-1}}$ are intervals of $f^p$; let $g_{k_0},\ldots, g_{k_{p-1}}$ be their corresponding affine maps. To decide whether $t$ is reachable, first check whether it belongs to one of the intervals $I_{k_0},\ldots, I_{k_{p-1}}$, if it does not, then $t$ is not reachable, if it does belong to, say, the interval $I_j$, then the problem is reduced to the question of whether $t$ is reachable from $f^j(x_0)\in I_j$, using the affine map $g_j$, for which we can use (i). 
  \end{proof}

  In the rest of this section, for the proof of \Cref{th:injective}, we assume that $f$ is an \emph{injective} \pam with two intervals:
  \begin{align*}
    I_1\defeq [0,c),\qquad\qquad I_2\defeq [c,1),
  \end{align*}
  for some cutpoint $c\in\rat$, $0<c<1$, and
  \begin{align*}
    f_1(x)\defeq a_1x+b_1,\qquad\qquad f_2(x)\defeq a_2x+b_2,
  \end{align*}
  the corresponding affine maps. By injectivity, $f(I_1)$ is disjoint from $f(I_2)$. If $f(I_1)< f(I_2)$, then it is not difficult to see that $G_f$ cannot have any loops that are not self-loops. In this case, we can decide reachability with \Cref{lem:easy}. So then, let us assume that:
  \begin{align}
    \label{eq:twist}
    f(I_2)<f(I_1). 
  \end{align}
  \subsection{Negative Slopes}
  \label{subsec: negative slopes}
  The next step towards the proof of \Cref{th:injective} is to treat the case when at least one of the slopes is negative, \ie at least one of the $a_1,a_2$, defined above, is negative. We proceed with the following two lemmas:
  \begin{lemma}
    \label{lem:neg1}
    Suppose that $a_2<0$, and furthermore $c\in f_2(I_2)$. Then we can decide reachability for~$f$. 
  \end{lemma}
  \begin{proof}
    If $c$ is in the endpoint of $f_2(I_2)$, \ie if $f_2(c)=c$, then except for this fixed point, which can be handled separately, the interval reachability graph of $f$, $G_f$, in this case looks like:
    \begin{align*}
      1\rightleftarrows 2
    \end{align*}
    and reachability can be decided by applying the second part of \Cref{lem:easy}. So assume that $c$ is in the interior of $f_2(I_2)$, and set
    \begin{align*}
      c'\defeq f_2^{-1}(c).
    \end{align*}
    \begin{center}
      \mfbox{  
        \begin{minipage}{0.4\textwidth}
          \includegraphics[width=\textwidth]{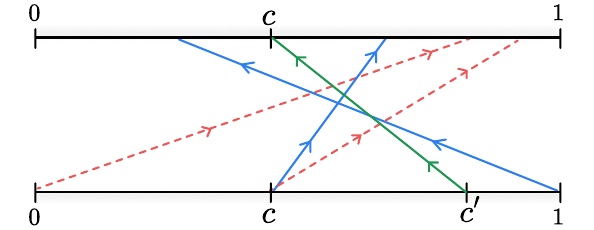}
        \end{minipage}
      }
    \end{center}    
    We define the following \pam:
    \begin{align*}
      g(x)\defeq\begin{cases}
                  f_2(x) & \text{ if }x\in [c,c')\\
                  f_1(f_2(x)) & \text{ if }x\in [c',1).
                \end{cases}
    \end{align*}
    Observe that $g$ is a \pam from $I_2=[c,1)$ to itself. Indeed, since $f_2$ has a negative slope, $a_2<0$, we have:
    \begin{align*}
      f_2([c,c'))&>c, \text{and}\\
      f_2([c',1))&\le c.
    \end{align*}
    And since we assumed \eqref{eq:twist}, it follows that
    \begin{align}
      \label{eq:to the right}
      f_1\bigg(f_2\big([c',1)\big)\bigg)\ge f_2\big([c,c']\big)>c,
    \end{align}
    so both components of $g$ are mapping to $I_2$.

    The idea of this new \pam $g$ is that reachability questions of $f$ can be reduced to those of $g$: Indeed if both $x_0$ and $t$ belong to the second interval $I_2$, then by construction of the map $g$, the point $t$ can be reached from $x_0$ by applying $f$ if and only if, it can be reached by applying $g$. If $x_0$ does not belong to $I_2$, then by injectivity of $f$, we can simply consider whether $t$ is reached from $f(x_0)$ which is guaranteed to be in $I_2$. The same holds when $t\notin I_2$. 

    Finally we show that reachability for the map $g$ is easy to decide. To see this, observe that \eqref{eq:to the right} implies that $G_g$ can have one of the two following forms:
    \begin{align*}
      \lcirclearrowright 1 \to 2 \rcirclearrowleft\qquad\qquad or \qquad\qquad \lcirclearrowright 1 \leftarrow 2 \rcirclearrowleft,
    \end{align*}
    in either case the only loops are self-loops, so we can decide reachability by appealing to \Cref{lem:easy}. 
  \end{proof}
  The proof of the next lemma is similar to the one above, except that the process of identifying the simplified \pam $g$ is repeated a (finite) number of times.
  \begin{lemma}
    \label{lem:neg2}
    Suppose that  $a_2<0<a_1$, and furthermore $c\in f_1(I_1)$. Then we can decide reachability~for~$f$. 
  \end{lemma}
  \begin{proof}
    Define the following two quantities:
    \begin{align*}
      c_1\defeq f_1^{-1}(c),\qquad\qquad m_0\defeq \min f_1(I_1),
    \end{align*}
    and consider the piecewise affine map $g_1$, defined as:
    \begin{align*}
      g_1(x)\defeq\begin{cases}
                    g_{1,1}(x):=f_1(x) &\text{ if }x\in[0,c_1),\\
                    g_{1,2}(x):=f_2(f_1(x)) &\text{ if }x\in[c_1,c). 
                  \end{cases}
    \end{align*}
    Using Assumption~\eqref{eq:twist}, we see that $g_1$ is a \pam from the interval $[0,c)$ to itself, with intervals $J_{1,1}:=[0,c_1)$ and $J_{1,2}:=[c_1,c)$. As in the proof of the preceding lemma,  if we can decide reachability for $g_1$, then we can do the same for $f$. To see this, note that if $x_0,t\in [0,c)$, then we clearly have:
    \begin{align*}
      \exists n\ f^n(x_0)=t\qquad\Leftrightarrow\qquad\exists m\ g_1^m(x_0)=t
    \end{align*}
    If $t$ does not belong to $[0,c)$ then by injectivity of $f$ we can simply consider the problem of reaching $f_2(t)$, starting from $f_2(x_0)$ or $x_0$, depending on whether $x_0\in[0,c)$ or not.  

   The new \pam $g_1$ inherits most properties from $f$; indeed, the slope of $g_{1,1}$ is positive, while that of $g_{1,2}$ is negative, the cutpoint of $g_1$ is $c_1$, and crucially
    \begin{align*}
      m_0=m_1\defeq\min g_{1,1}(J_{1,1}),
    \end{align*}
    due to $f_1$ having a positive slope.
    The only property it might not inherit is if $c_1$ does not belong to $g_{1,1}(J_{1,1})$. If it does not belong to $g_{1,2}(J_{1,2})$ either, then the interval reachability graph of $g_1$, $G_{g_1}$ has the form:
    \begin{align*}
      1\rightleftarrows 2
    \end{align*}
    and reachability in $g_1$ (and hence also in $f$) can be decided in this case due to the second part of \Cref{lem:easy}. If $c_1\in g_{1,2}(J_{1,2})$, then $g_1$ satisfies all the conditions of \Cref{lem:neg1}, as a consequence of which, we can again decide reachability in $g_1$.

    So suppose that $g_1$ also inherits the property that the cutpoint $c_1$ falls in the image of the first affine map $g_{1,1}$, \ie $c_1\in g_{1,1}(J_{1,1})$. Now we iterate the process, and define $c_2:=g_{1,1}^{-1}(c_1)$ and the new \pam $g_2$ as
    \begin{align*}
      g_2(x)\defeq \begin{cases}
        g_{2,1}(x):=g_{1,1}(x) &\text{ if }x\in [0,c_2),\\
        g_{2,2}(x):=g_{1,2}\big(g_{1,1}(x) \big) &\text{ if }x\in [c_2,c_1).
        \end{cases}
      \end{align*}
      which maps $[0,c_1)$ to itself. If again $c_2\in g_{2,2}(J_{2,2})$, we define $g_3$ and so on. We summarise the relevant properties of the sequence $f,g_1,g_2,\ldots$ of \pam. For all $k\in\nat$ we have:
      \begin{align}
        \label{eq:prop1}
        &g_{k,1}=g_{k-1,1}=\cdots =g_{1,1}=f_1,\\
        \label{eq:prop2}
        &c_k=f_1^{-1}(c_{k-1})=f^{-k}(c),\\
        \label{eq:prop3}
        &m_0=\cdots =m_k\defeq \min g_{k,1}(J_{k,1}).
      \end{align}
      Property \eqref{eq:prop1} is by definition, \eqref{eq:prop2} follows from \eqref{eq:prop1} and the fact that in the sequence $g_1,g_2,\ldots$ of \pam we have assumed that the cutpoint belongs to the image of the first map. Property \eqref{eq:prop3} follows from \eqref{eq:prop1} and the fact that this minimum is reached at $0$. 

      We claim that this iterative process halts, \ie there exists $K\in\nat$ such that $c_K$ does not belong to the image of the first map $g_{K,1}$,
      \begin{align*}
        c_K\notin g_{K,1}(J_{K,1}). 
      \end{align*}
      Indeed as a consequence of \eqref{eq:prop3} we have:
      \begin{align}
        \label{eq:gk bounded}
        g_{k,1}(J_{k,1})\ge m_0. 
      \end{align}
      Since the function $f_1$ is monotone increasing in $[0,c)$ and $f_1(0)>0$, $f_1(c)>c$, it follows that for any $x\in f_1([0,c))$ there is some $N_x\in\nat$ such that
      \begin{align*}
        f_1^{-N_x}(x) < f_1(0)=m_0.
      \end{align*}
      Due to \eqref{eq:prop2}, this means that there is some $K\in\nat$, such that $c_K<m_0$, and from \eqref{eq:gk bounded}, $c_K$ does not belong to the image of $g_{K,1}$, \ie $c_K\notin g_{K,1}(J_{K,1})$. The claim is proved.

      As before, deciding reachability in $g_K$ is easy, since either the interval reachability graph $G_{g_K}$ is of the form that allows us to apply \Cref{lem:easy}, or the conditions of \Cref{lem:neg1} are fulfilled. And if we can decide reachability in $g_K$, then we can do the same in $g_{K-1}$, and so on, for all the maps in the sequence $f,g_1,g_2,\ldots, g_K$. In particular this means that we can decide reachability in $f$. 
    \end{proof}

    The two lemmas above, in fact, cover all the cases where at least one of the slopes is negative. This claim can be proved as follows. Suppose that both $f_1$ and $f_2$ have negative slopes. If the cutpoint $c$ is not in the image of either $f_1$ or $f_2$ (\ie $c\notin f_1(I_1)$, and $c\notin f_2(I_2)$) then in the interval reachability graph every vertex has a unique outgoing edge, hence we can use \Cref{lem:easy}. If $c$ belongs to the image of $f_2$ then we use \Cref{lem:neg1}. If $c$ belongs to the image of $f_1$ on the other hand, we apply the bijection from $[0,1)$ to $(0,1]$, given by the map $h(x)=1-x$, to define a new \pam $f':=h^{-1}\circ f \circ h$, whose cutpoint belongs to the image of the second map $f'_2$, and hence \Cref{lem:neg1} is applicable.

    Similarly, if $f_1$ has a negative slope, but $f_2$ a positive one, we turn things around with the function $1-x$ and apply the two lemmas above. We have proved:
    \begin{proposition}
      The reachability problem is decidable for injective \pam with two intervals, such that at least one of the affine maps has a negative slope. 
    \end{proposition}
    Having dealt with this case, in the rest of the current section we assume that both slopes are positive. We split the proof into the two cases, depending on whether $f$ is surjective or not. 
    \subsection{Bijections}
    \label{subsec:bij}
    In \cite{boshernitzan93_dense_orbit_ration}, Boshernitzan considers the question of whether there exists a \pam in whose definition only rational numbers figure, that also has orbits which are dense in $\unit$? The question in answered positively by exhibiting a family of such examples, that are shown to be topologically similar to rotations of $\rel/\intg$ by an angle that is not a rational multiple of~$\pi$. In this subsection, by utilising Boshernitzan's approach, we demonstrate that it not only helps with investigating topological properties of orbits, but it also aids us to decide reachability.

    We are given the injective \pam $f$ with intervals $I_1,I_2$, such that $f(I_2)<f(I_1)$, \ie \eqref{eq:twist}, and furthermore, thanks to the previous subsection, we assume that the slopes of the two affine components $f_1,f_2$ are both positive. In this subsection, let us further suppose that $f$ is surjective. Such maps are characterised by exactly two rational numbers in $(0,1)$; indeed they map some interval $[0,c)$ to some interval $[d,1)$, and they map $[c,1)$ to $[0,d)$, so they are characterised by the rationals $c,d \in (0,1)$. They look as follows:
    \begin{center}
      \mfbox{  
        \begin{minipage}{0.4\textwidth}
          \includegraphics[width=\textwidth]{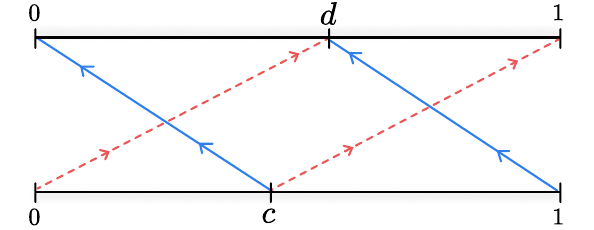}
        \end{minipage}
      }
    \end{center}    
    Explicitly, they can be defined as:
    \begin{align*}
      f(x)\defeq\begin{cases}
        \frac{1-d} c x + d &\text{ if }x\in [0,c),\\
        \frac d {1-c} x - \frac{cd}{1-c} &\text{ if }x\in [c,1),
      \end{cases}
    \end{align*}
    for any rationals $c,d$ in $(0,1)$.

    For any $\tau>0$, denote by $R_\tau$ the \textbf{rotation} by $\tau$ in $\rel/\intg$, that is:
    \begin{align}
      \label{eq:rotation}
      R_\tau(x)\defeq x+\tau\ \mathrm{mod}\ \intg,
    \end{align}
    where by the expression on  the right-hand side we mean: add $\tau$ to $x$ and then take the fractional part of the result. 
    
    To begin with, consider a \pam as defined above, where $c+d=1$. We claim that $f=R_d$. Indeed, replacing $1-d$ by $c$ and $1-c$ by $d$ in the definition of $f$ we see that the first affine component is $x+d$ while the second one is $x-c=x+d-1$, therefore $f=R_d$. So reachability from $x_0$ to $t$, in this case, is asking whether there exists some $n$ such that
    \begin{align*}
      x_0+nd\equiv t\mod\intg,
    \end{align*}
    or equivalently whether there are positive integers $n,m$ such that
    \begin{align*}
      x_0+nd=t+m,
    \end{align*}
    which is easily seen to be decidable.

    If $c+d>1$, we define a new \pam by applying the bijection $h(x)=1-x$, so that everything is rotated around, and in particular for the corresponding rationals $c',d'$ we have $c'+d'<1$. So assume that $c+d<1$, and define the quantity:
    \begin{align}
      \label{eq:def alpha}
      \alpha\defeq \frac {1-c-d} {cd}>0,
    \end{align}
    and the function $h$,
    \begin{align}
      \label{eq:def h}
      h(x)\defeq \frac{\log(\alpha x+1)}{\log(\alpha+1)}. 
    \end{align}
    Since $\alpha>0$, $h(x)$ is well-defined for any real $x\ge 0$. We prove that $h$ is a continuous bijection from $\unit$ to itself, that has an inverse that is itself continuous, in other words $h$ is a homeomorphism. Indeed we see that $h(0)=0$ and $h(1)=1$, and since the derivative
    \begin{align*}
      h'(x)=\frac \alpha {(\alpha x+1)\log(\alpha+1)}
    \end{align*}
    is positive in $[0,1]$, due to the fact that $\alpha>0$, we conclude that $h$ is a bijection, and clearly both $h$ and its inverse are continuous. Define now the angle $\tau$ as:
    \begin{align*}
      \tau\defeq h(d) = \frac{\log(\alpha d+1)}{\log (\alpha+1)}. 
    \end{align*}
    We show that $f$ is \emph{topologically similar} to a rotation by $\tau$, that is:
    \begin{lemma}
      \label{lem:similar to a rotation}
      $f = h^{-1} \circ R_\tau \circ h.$
  \end{lemma}
  \begin{proof}
  Define first:
    \begin{align*}
      R'_\tau(x)\defeq x+\tau,\qquad R''_\tau(x)\defeq x+\tau-1,
    \end{align*}
    and note that
    \begin{align*}
      h^{-1}(x)=\frac {e^{x\log(\alpha+1)}-1} \alpha.
    \end{align*}
    Let $x\in [0,c)$, we show that $R'_\tau\circ h(x)<1$. We have
    \begin{align*}
      R'_\tau\circ h(x)&=R'_\tau(h(x))=h(x)+\tau\\
      &=h(x)+h(d)=\frac {\log[(\alpha x+1)(\alpha d+1)]} {\log(\alpha+1)}. 
    \end{align*}
    To prove that the quantity above is strictly smaller than $1$, observe that $(\alpha x+1)(\alpha d+1)>1$ and $\alpha+1>1$, so since $\log$ is monotone increasing in $[1,\infty)$ it suffices to show that:
    \begin{align*}
      (\alpha x+1)(\alpha d +1)<\alpha+1.
    \end{align*}
    One can prove that the inequality above holds by noting that the left-hand side is strictly smaller than $(\alpha c+1)(\alpha d+1)$, which is equal to the right hand side. Now since $R'_\tau\circ h(x)<1$ we see that $R_\tau\circ h=R'_\tau\circ h$ for $x\in[0,c)$. Finally, a simple computation shows that for $x\in [0,c)$ we have
    \begin{align*}
      h^{-1}\circ R_\tau\circ h(x)=\frac {(\alpha x+1)(\alpha d+1)-1} \alpha = \frac {1-d} c x + d,
    \end{align*}
    which proves that the statement of the lemma holds for $x$ in $[0,c)$. If $x\in [c,1)$, from the same argument as above we can see that $R'_\tau\circ h(x)\ge 1$, so $R_\tau\circ h(x)=R''_\tau\circ h(x)$. Hence
    \begin{align*}
      R_\tau\circ h(x) &= \frac {\log(\alpha x+1)} {\log(\alpha+1)}+\frac {\log(\alpha d+1)} {\log(\alpha+1)}-1\\
                       &=\frac {\log\left[\frac {(\alpha x+1)(\alpha d+1)}{\alpha+1}\right]} {\log (\alpha+1)}
    \end{align*}
    After applying $h^{-1}$ to the quantity above, a simple calculation shows that
    \begin{align*}
      h^{-1}\circ R_\tau \circ h(x)=\frac d {1-c}x-\frac {cd} {1-c},
    \end{align*}
    for $x\in [c,1)$. 
  \end{proof}
  Now we show how this lemma can be used to decide whether $t\in\unit$  is reached from $x_0\in\unit$ by applying $f$. As a consequence of the lemma and $h$ being a bijection we have that for all $n\in\nat$
  \begin{align*}
    f^n(x_0)=h^{-1}\circ R_\tau^n\circ h(x_0),
  \end{align*}
  whence, the quantity above is equal to $t$ if and only if there exists some $n\in\nat$ such that
  \begin{align*}
    h(t)=R_\tau^n\big(h(x_0)\big). 
  \end{align*}
  Which is equivalent to the question of whether there exists some $n\in\nat$ such that
  \begin{align*}
    h(t)\equiv h(x_0)+n\tau\mod\intg,
  \end{align*}
  which amounts to asking whether there are positive integers $n,m$ such that
  \begin{align*}
    m+h(t)=h(x_0)+n\tau.
  \end{align*}
  The equation above, through a simple calculation, is shown to hold if and only if one can find positive integers $n,m$ such that
  \begin{align*}
    m=\frac {\log \left[\frac{(\alpha x_0+1)(\alpha d+1)^n}{(\alpha t+1)}\right]} {\log (\alpha+1)}.
  \end{align*}
  Since $\alpha>1$ and $\log$ is injective, the equation above holds if and only if
  \begin{align*}
    (\alpha+1)^m=\frac{\alpha x_0+1}{\alpha t+1} (\alpha d+1)^n. 
  \end{align*}
  Every factor in the equation above is a rational number, and consequently we can decide if there are positive integers $n,m$ such that the equation holds. This terminates the proof of decidability when $f$ is a bijective \pam with two intervals whose slopes are positive. Now we treat the case in which $f$ is not surjective. 
  \subsection{Maps with Gaps}
  \label{subsec:gaps}
  Suppose that the injective \pam $f$ with two intervals and positive slopes is \emph{not} surjective, so it maps $[0,c)$ to some $[a_1,b_1)$ and $[c,1)$ to some $[a_2,b_2)$. By injectivity and \eqref{eq:twist}, $b_2\le a_1$.
      \begin{center}
      \mfbox{  
        \begin{minipage}{0.4\textwidth}
          \includegraphics[width=\textwidth]{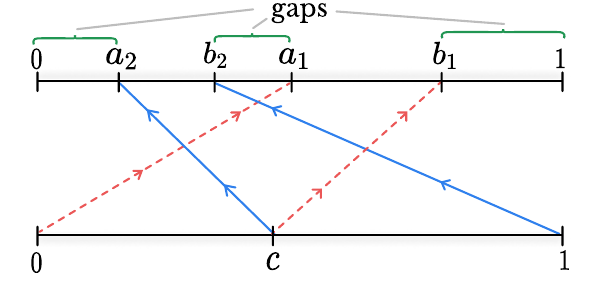}
        \end{minipage}
      }
    \end{center}    

  Since it is not surjective, at least one of the following strict inequalities has to hold: $0<a_2$ (there is a gap to the left of the first image), $b_2<a_1$ (a gap between the two images), $b_1<1$ (a gap to the right of the second image). If there is a gap on the sides, \ie if at least one of the inequalities $0<a_2$, $b_1<1$, hold, we proceed as follows.

  If $c$ does not belong to the interval $I:=[a_2,b_1)$ then the interval reachability graph $G_f$ is such that every vertex has a single outgoing edge, and hence we can decide reachability thanks to \Cref{lem:easy}. Suppose that $c\in I$, and consider $f'$ the restriction of $f$ to the set $I$. Since the slopes are positive, clearly we see that $f'$ is a \pam from the interval $I$ to itself. If the target $t$ is not in $I$, then it clearly cannot be reached. If the starting point $x_0$ is not in $I$, then $f(x_0)$ is in $I$. So the question of whether $t$ is reachable from $x_0$ via $f$, can be reduced to the question of whether $t$ is reachable from $f(x_0)$ via the new map $f'$. As in the beginning of the present section, we can scale $I$ to the unit interval $\unit$ and define another \pam $f''$, which will not have any side gaps.

  The only remaining case is of a \pam that maps $[c,1)$ to some $[0,a)$ and $[0,c)$ to some $[b,1)$. This family of \pam has been studied by Laurent and Nogueira in the recent paper~\cite{LGarxiv, LG21}. There the authors give a rather explicit description of the dynamics of these maps, which we will exploit to decide reachability. Analogously to the preceding section, Laurent and Nogueira, construct a map using a Hecke-Mahler type series, which allows one to view the given \pam as a certain rotation. 

  Clearly, the \pam under consideration are characterised by the three rationals $a,b$ and $c$. In order to unify the notation with that of~\cite{LG21}, we note that the same family can be characterised by three different rationals, namely the slope of the first affine component, denoted $\lambda>0$, the translation of the same denoted $\delta$ and a positive real $\mu>0$ involved in the definition of the second component. To be more precise, the family of \pam that we consider in this last subsection can be defined by three real numbers $\lambda,\mu,\delta$ with the following properties:
  \begin{align*}
    &0 < \lambda < 1,\\
    &0 < \mu,\\
    &1-\lambda < \delta < d_{\lambda,\mu}\defeq
      \begin{cases}
        1 & \text{ if }\lambda\mu<1,\\
        \frac{\mu-\lambda\mu}{\mu-1} & \text{ if }\lambda\mu\ge 1. 
      \end{cases}
  \end{align*}
  So that a \pam $f_{\lambda,\mu,\delta}$ from $\unit$ to $\unit$ is defined as follows, let $c:=(1-\delta)/\lambda$ and
  \begin{align*}
    f_{\lambda,\mu,\delta}(x)\defeq
    \begin{cases}
      \lambda x+\delta & \text{ if }x\in [0,c),\\
      \mu(\lambda x+\delta-1) & \text{ if }x\in [c,1). 
    \end{cases}
  \end{align*}
  Note that we have not lost any generality by assuming that $\lambda<1$ due to the following reason. Since $f$ is injective but not surjective, one of the affine components must be contracting, \ie it must have slope $<1$. If it is not the first affine component (as we have assumed here with $\lambda<1$), then we can permute the components by applying the bijection $h(x)=1-x$ as was done above. 
  
  Briefly, the dynamics look as follows. There is a unique cyclic orbit $C=\set{c_0,\ldots,c_{k-1}}$~\cite[Theorem 3]{LG21}, so $f$ sends $c_i$ to $c_{i+1\ \mathrm{mod}\ k}$. All other orbits are infinite, but they have $k$ accumulation points, namely the points in $C$. In other words, all other orbits approximate the unique cyclic orbit $C$.

  Our first step is to decide whether the starting point $x_0$ or the target $t$ belong to $C$. For this we need to effectively determine an upper bound on the length of the cycle, which can be done by computing the \emph{rotation number} of $f$. The rotation number $\rho$ of $f$ is rational when the constants $\lambda,\mu,\delta$ are rational, due to a classical transcendence result~\cite[Theorem 3]{LG21}. Furthermore the rotation number $\rho$ is equal to $\rho=p/q$ for some positive and co-prime integers $p,q$ if and only if
  \begin{align*}
    F_1(\lambda,\mu,\frac p q) \le \delta < F_2(\lambda,\mu,\frac p q),
  \end{align*}
  where $F_1,F_2$ are a pair of algebraic functions given explicitly in~\cite[Theorem 3]{LG21}. From these facts, it is plain that we can compute $\rho$. Indeed, take any enumeration of $\rat$ and for each element decide whether the inequalities above hold; such a procedure must halt, because we know that there exists a rational $\rho=p/q$ for which the inequalities above hold. To speed up this process, we may search among the Farey sequence and use the following upper bound on $\rho$:
  \begin{align*}
    r_{\lambda,\mu}\defeq
    \begin{cases}
      1 & \text{ if }\lambda\mu<1,\\
      \frac {\log \frac 1 \lambda} {\log \mu} & \text{ if }\lambda\mu\ge 1.
    \end{cases}
  \end{align*}

  Let $\rho=p/q$ be the computed rotation number, then part (ii) of ~\cite[Theorem 3]{LG21} implies that the unique cycle $C$ has length $q$. Compute the rational numbers:
  \begin{align*}
    f(x_0),f^2(x_0),\ldots,f^{q}(x_0).
  \end{align*}
  If one of the entries in this finite sequence is equal to $x_0$, then $x_0$ belongs to $C$. Similarly compute
  \begin{align*}
    f(t),f^2(t),\ldots,f^{q}(t),
  \end{align*}
  and see whether one of the entries is equal to $t$ to ascertain whether $t\in C$. If both the starting point $x_0$ and the target are in $C$, then clearly the target can be reached. If one of them is in $C$ but the other one is not, then $t$ cannot by reached from $x_0$, due to injectivity of $f$. In the case when neither $x_0$ or $t$ is in $C$, we will effectively compute a threshold $N\in\nat$ after which the target $t$ cannot be reached.
  \begin{lemma}
    \label{lem:threshold}
  If neither the starting point $x_0$ nor the target $t$ belong to the cycle $C$ then we can effectively determine a threshold $N\in\nat$ such that for all $n>N$ we have
  \begin{align*}
    t\ne f^{n}(x_0). 
  \end{align*}
\end{lemma}
\begin{proof}
Let $\rho=p/q$ be the rotation number computed above. We explain how we can compute this threshold. To do that, we need to first have a better understanding of the points that are in the cycle $C$. They are given via the Hecke-Mahler series; more precisely the cycle $C$ is equal to the set of points $\Phi(k/q)$, $0\le k\le q-1$, where the function $\Phi$ is given by the series:
  \begin{align*}
    &\Phi(x)\defeq \lfloor x\rfloor+ \frac {1-\delta} \lambda \\ 
           &+ \sum_{n\ge 0}\lambda^n\mu^{\lfloor x\rfloor - \lfloor x-n\rho\rfloor}\bigg(\frac {\lambda+\delta-1} \lambda +\lfloor x-(n+1)\rho\rfloor -\lfloor x-n\rho\rfloor\bigg).
  \end{align*}
  This function converges when $\lambda\mu^\rho<1$, which holds for our case. Immediately, we see that given any $\epsilon>0$ and any positive rational number $r\in\rat$, $r>0$, we can compute an interval $I$ of length $\epsilon$ such that $\Phi(r)\in I$. In other words, we can approximate $\Phi(r)$ to arbitrary additive precision. Compute an $\epsilon>0$ (by \eg trying $2^{-n}$ for larger and larger $n$) with the following property. For $\epsilon$ and $\Phi(k/q)$, $0\le k \le q-1$, using the approximation above compute the interval $J_k$ of length $\epsilon$, such that the distance from the target $t$ and any interval $J_k$, $0\le k \le q-1$, is at least $2\epsilon$. Such an $\epsilon$ exists due to the fact that $t$ is not equal to any $\Phi(k/q)$ since it does not belong to the cycle $C$. Note that by construction of $\epsilon$, if we take \emph{any} other intervals $J'_k$ of length $\epsilon$ containing $\Phi(k/q)$, $0\le k\le q-1$, $t$ will be outside all of them. 

  Now the following two facts will conclude the proof of the lemma: For all $n\in\nat$,
  \begin{align}
    \label{eq:fact 1}
    &f^n(\unit) \text{ has measure smaller than }(\lambda^q\mu^p)^{\lfloor n/q\rfloor},\text{ and }\\
    \label{eq:fact 2}
    &f^n(\unit)=\biguplus_{k=0}^{q-1}H_k, 
  \end{align}
  where $H_k$ is an interval containing $\Phi(k/q)$, and the measure is the Lebesgue measure. 
  
  The statement \eqref{eq:fact 1} is from Proposition~6 in~\cite{LG21}, while statement \eqref{eq:fact 2} is in the succeeding corollary. Choose an $N\in\nat$ such that
  \begin{align*}
    (\lambda^q\mu^p)^{\lfloor N/q\rfloor}<\epsilon.
  \end{align*}
  From \eqref{eq:fact 2} we see that for all $n>N$, $f^n(\unit)$ is made out of $q$ disjoint intervals $H_k$, each one of length smaller than $\epsilon$, consequently by construction of $\epsilon$, we have
  \begin{align*}
    t\notin f^n(\unit).
  \end{align*}
\end{proof}
Thus we have concluded the proof of the main theorem, \Cref{th:injective}. 
\section{Related Decision Problems}
\label{sec:related problems}

A more careful examination of the proof of \Cref{th:injective} shows that the same can be used to demonstrate decidability of certain variants of the reachability problem. For example, the \textbf{point-to-interval} reachability: given an interval $I\subset\unit$ and $x_0$ decide whether there exists some $n\in\nat$ such that $f^n(x_0)\in I$. First, we observe that for general \pam this problem is no harder than the point-to-point reachability considered in the prequel.
\begin{proposition}
  \label{prop:point to interval}
  The point-to-interval reachability problem can be effectively reduced to the point-to-point reachability problem. 
\end{proposition}
\begin{proof}
  Let $f$ be a \pam and $I_1,\ldots, I_\ell$ its intervals. Let $x_0$ be the initial point and $I$ the given target interval. We effectively construct another \pam $f'$ as follows. The intervals $I_j$ of $f$ that do not intersect the target interval $I$ are also intervals of $f'$ with the same associated affine map. For all intervals $I_j$ that intersect $I$ but are not contained in it, $f'$ has the interval $I_j\setminus I$ with the affine map corresponding to $I_j$. For $I$ the associated affine map is the constant map $g(x)=t$ for some rational $t>1$.

  It follows that the orbit of $x_0$ under $f$ intersects the target interval $I$ if and only if the point $t$ belongs to the orbit of $x_0$ under $f'$. We can now scale the \pam $f'$ to make it a map from the unit interval to itself. 
\end{proof}
We cannot directly use \Cref{prop:point to interval} to prove that point-to-interval reachability in injective \pam with two intervals is decidable, because in the reduction the number of intervals is increasing. However we can show that a procedure exists by applying the analysis that was done in the preceding section. Here is a sketch of the proof. 

\begin{theorem}
  \label{th:point to interval}
  The point-to-interval reachability problem is decidable for injective \pam with two intervals. 
\end{theorem}
\begin{proof}[Proof Sketch]
  Let $f$ be the given \pam, $x_0\in\unit$ the initial point, and $I$ the target interval. When the interval reachability graph of $f$ is particularly simple, as we saw in \Cref{lem:easy}, the reachability problem reduces to that of a single affine map. Deciding whether it is possible to reach an interval, in this case, is trivial. The case of negative slopes in \Cref{subsec: negative slopes} was essentially reduced to that of \pam covered in \Cref{lem:easy}.

  It remains to consider \pam that are bijections (of \Cref{subsec:bij}) and those that have a gap in the middle (of \Cref{subsec:gaps}). For the former, we proved that such \pam are topologically conjugate to rotations in the circle by a quantity:
  \begin{align*}
    \tau=\frac{\log(q_1)}{\log (q_2)},
  \end{align*}
  where $q_1,q_2$ are rational numbers depending on $f$. It is a basic theorem that
  \begin{align*}
    \set{n\tau\ \mathrm{mod}\ \intg\st n\in\nat}
  \end{align*}
  is dense in $\unit$ if and only if $\tau$ is irrational. If $\tau$ is rational, on the other hand, the set above is finite. Hence, if $\tau\notin\rat$ then every orbit of $f$ is dense in $\unit$, therefore every interval is reached from every point. If $\tau\in\rat$ however, then every orbit of $f$ is periodic, so we would only need to compute finitely many iterations of $f$ and check whether any of them send $x_0$ to $I$. Finally to decide whether $\tau$ is rational we proceed as follows. We observe that $\tau\in\rat$ if and only if there are integers $a,b$ such that
  \begin{align*}
    q_1^a=q_2^b,
  \end{align*}
  which one can decide easily by looking at the prime factorisation of $q_1,q_2$.

  In maps that have gaps in the middle, from \cite{LG21}, we know that there is a single periodic orbit with points $c_0,\ldots, c_{k-1}$ whose length $k$ we can compute. First we check whether the initial point $x_0$ belongs to this unique periodic orbit, by simply computing the first $k$ entries of the orbit. If it does, then clearly $I$ is reached if and only if one of the first $k$ entries belongs to $I$. Similarly, then check whether one of the endpoints of $I$ belongs to the unique periodic orbit. If not, one can proceed as in \Cref{lem:threshold}, by approximating $c_i$ to either find one that sits inside the interior of $I$, or to compute some threshold $N\in\nat$, after which we know that $f^n(x_0)$ is not in the target $I$. 

  The remaining case is when one of the endpoints of $I$ (call it $c$) belongs to the periodic orbit, but none of the $c_i$ are in the interior of $I$.  It follows from \eqref{eq:fact 1} and \eqref{eq:fact 2} that $f^{nk}(x_0)$, $n\in\nat$ is a Cauchy sequence tending to $c$, and that $c$ is a fixed point of the \pam $f':=f^k$. If $c$ belongs to the interior of one of the intervals defining $f'$, then the problem reduces to a question about a single affine map, and is easily dealt with. If $c$ is the point between the adjacent intervals $J_1,J_2$ in the definition of $f'$, by looking at the corresponding affine maps, one can decide whether both $J_1$ and $J_1$ are visited infinitely often or the orbit stays in only one of them, and decide accordingly. 
\end{proof}

Since in the two more complicated families of \pam the initial point does not play a big role, meaning that most orbits look the same, bar a few small modifications, one can use the proof above to also show that \textbf{interval-to-interval} reachability is decidable for injective \pam with two intervals. That is the decision question where one is given a \pam $f$ with two intervals that is injective, a starting interval $J_0$ and a target interval $J_1$ and is asked to decide whether there exists some $x_0\in J_0$ and natural $n\in\nat$ such that $f^n(x_0)$ is in $J_1$.

Another interesting problem whose decidability comes as a corollary from \Cref{sec:reachability}, is to decide whether the orbit is \textbf{periodic}. More precisely, given a \pam $f$ and $x_0\in\unit$ decide whether there exists two distinct naturals $n,m\in\nat$ such that $f^n(x_0)=f^m(x_0)$.

\begin{theorem}
  The periodic orbit problem is decidable for injective \pam with two intervals. 
\end{theorem}
\begin{proof}[Proof Sketch]
  In cases that were covered by \Cref{lem:easy} the orbit can be periodic only if the fixed point of one of the affine maps $g(x)=ax+b$ is reached (when every vertex of $G_f$ has a unique outgoing edge, a slight variation of this idea applies). That fixed point, if it exists, is equal to $b/(1-a)$. So the problem reduces to deciding whether there exists some $n\in\nat$ such that
  \begin{align*}
    g^n(x_0')=a^nx_0'+\frac{a^n-1}{a-1} b = \frac {b} {1-a},
  \end{align*}
  where $x_0'$ can be effectively computed.

  When $f$ is a bijection, in the proof of \Cref{th:point to interval} we saw that the orbit is finite if and only if the quantity $\tau$ is rational, and that this can be decided. When $f$ has a gap in the middle of the two images, we saw in \Cref{subsec:gaps} that there is a unique periodic orbit, whose length we can compute. 
\end{proof}

We conclude this section with a surprising result; we show that a problem related to reachability is expected to be rather difficult to decide, even for injective \pam with two intervals. This is the problem of deciding whether there exists some $n\in\nat$ such that $nf^n(x_0)<c$ for some given $c\in\unit$. So it is asking whether the orbit ever goes inside an interval that is shrinking with time, \ie the interval $[0,c/n)$. Let us call it the \textbf{shrinking interval} problem. The reason we believe this problem should be difficult to decide is because such a procedure would expose a lot of information about the Diophantine approximation properties of certain transcendental numbers, and would consequently answer a number of open problems.
\subsection{Shrinking Interval Problem}
Understanding how well an irrational number $\mu$ can be approximated by rationals is important for the solving Diophantine equations, and other central problems in number theory. That is the question of how small can the quantity
\begin{align*}
  \left|\mu-\frac p q\right|
\end{align*}
be as a function of $q$, when $p,q$ range over integers. As a consequence of the classical theory of continued fractions (see~\cite[Chapter II, Section 8]{khinchin}), it is known that for every real number $\mu$, there are integers $p,q\in\intg$ such that
\begin{align*}
  \left|\mu-\frac p q\right|\le \frac 1 {q^2}. 
\end{align*}
With what smaller constant can we replace $1$ and have the inequality above hold? In other words what can we say about:
\begin{align*}
  L(\mu)\defeq \inf \set{c\in\rel\st \left|\mu-\frac p q\right|\le \frac c {q^2}, \text{ for some }p,q\in\intg}. 
\end{align*}
This number is sometimes called the \textbf{Lagrange constant} of~$\mu$, and it measures in a sense how well can the number $\mu$ be approximated by rationals; numbers that have $L(\mu)>0$ are called badly approximable. We still do not know almost anything about the Lagrange constants of specific numbers; most results we have are of metric nature (\eg almost all real numbers have Lagrange constant $0$), \cite[Chapter III]{khinchin}.

In this subsection we will show that a procedure for the shrinking interval problem, could be used to approximate $L(\mu)$, not for all reals $\mu$, but for a family of reals; namely those that belong to the following set:
\begin{align*}
  S\defeq\set{\frac 1 {1 + \frac {\log {\frac {1-c} d }} {\log \frac {1-d} c}} \st c,d\in\rat\cap (0,1), c+d<1}. 
\end{align*}

The positivity problem for linear recurrence sequences (\ie is there a positive term in the given sequence) has the same property \cite[Section 5]{ouaknine13_posit_probl_low_order_linear_recur_sequen}, namely that a procedure to decide the positivity problem could be used to approximate $L(\mu)$, however, for a $\mu$ in a different set (the arguments of algebraic numbers in the unit interval).

We find it surprising that the same phenomenon occurs even for one-dimensional, injective \pam with two intervals; one would expect that every natural problem is decidable for these systems.

\begin{theorem}
  If the shrinking interval problem is decidable for injective \pam with two intervals, then there exists a procedure that inputs $\mu\in S$, $\epsilon>0$ and computes an interval $I\subset [0,1]$ of length $\epsilon$ such that $L(\mu)\in I$. 
\end{theorem}
\begin{proof}
  We define the \pam whose shrinking interval problem exposes information about $L(\mu)$. To that end, let $c,d$ be two rationals in $(0,1)$, such that $c+d<1$ and $\mu\in S$ the corresponding real number (see the definition of $S$ above). Let $f$ be the bijective \pam with two intervals defined by the rationals $c$ and $d$ as in \Cref{subsec:bij}. In that section we saw that $f$ is topologically similar to a rotation by an angle $\tau$, which in our case a short calculation shows that $\mu=\tau$. Therefore, due to \Cref{lem:similar to a rotation}, we have
  \begin{align*}
    f=h\circ^{-1} R_\mu\circ h,
  \end{align*}
  where $R_\mu$ and $h$ are defined in \eqref{eq:rotation} and \eqref{eq:def h} respectively. From this equivalent definition of $f$, for any real $\gamma\in (0,1)$ and $n\in\nat$, we have 
  \begin{align*}
    n f^n(0) < \gamma\qquad \Leftrightarrow \qquad n\mu\  \mathrm{mod}\ \intg < h(\gamma/n),
  \end{align*}
  because $h$ is a monotone increasing homeomorphism. That statement holds if and only if there is a positive integer $m\in\nat$ such that
  \begin{align}
    \label{eq:yes ineq}
    \left|\mu-\frac m n\right| <h\left(\frac\gamma n\right)\frac 1 n. 
  \end{align}

  Denote by $\algo$ the hypothetical algorithm for deciding the shrinking interval problem. If $\algo(f,\gamma)$ returns \emph{yes}, then there exists natural numbers $n,m$ such that \eqref{eq:yes ineq} holds. If $\algo(f,\gamma)$ returns \emph{no} on the other hand, that implies that for all naturals $n,m$ we have
  \begin{align}
    \label{eq:no ineq}
    \left|\mu-\frac m n\right| \ge h\left(\frac\gamma n\right)\frac 1 n. 
  \end{align}
  Observe that if $\algo(f,\gamma)=\text{\emph{yes}}$ then for any larger $\gamma'$, again we have  $\algo(f,\gamma')=\text{\emph{yes}}$; and symmetrically if $\algo(f,\gamma)=\text{\emph{no}}$ then for any smaller $\gamma'$ we have $\algo(f,\gamma')=\text{\emph{no}}$. It plainly follows that for all $\delta>0$ we can compute an interval $J\subset \unit$ of length at most $\delta$, such that if $\algo(f,\cdot)$ changes answer from \emph{yes} to \emph{no}, it does so for some real $\gamma$ in the interval $J$. This computation is performed by making multiple calls to $\algo(f,\gamma)$ for different~$\gamma$.

  Furthermore, unfolding the definition of the homeomorphism $h$, using the Maclaurin series for $\log(x+1)$, and performing a simple calculation we see that
  \begin{align*}
    \frac {\alpha \gamma} {2 n^2} \le h\left(\frac \gamma n\right)\frac 1 n \le \frac {2\alpha \gamma} {n^2}, 
  \end{align*}
  where $\alpha$ is defined as in \eqref{eq:def alpha}, hence \eg, if $\algo(f,\gamma)=\text{\emph{yes}}$ then we have an upper bound for the Lagrange constant, namely $L(\mu)\le 2\alpha\gamma$. More generally, one can show that for any $\epsilon'>0$, and all $n$ sufficiently large we have
  \begin{align*}
    (1-\epsilon')\frac {\alpha \gamma} {n^2} \le h\left(\frac \gamma n\right)\frac 1 n \le (1+\epsilon')\frac {\alpha \gamma} {n^2}. 
  \end{align*}
  Since both $\delta$ and $\epsilon'$ can be taken to be arbitrarily small, the theorem follows. 
\end{proof}
\section{Orbits of Expanding \pam}
\label{sec:dichotomy}
In \cite[Hypothesis 1]{bournez18_reach_probl_one_dimen_piecew_affin_maps}, \cite[Hypothesis 1]{kurganskyy2016reachability}, the authors conjecture that expanding \pam (\ie those whose affine maps all have slope $>1$) have the property that for all $x_0\in\unit$,
\begin{align}
  \label{eq:hyp1}
  \orbit(f,x_0)\text{ is finite, or is dense in }\unit. 
\end{align}
For a subclass of expanding \pam, namely for $\beta$-expansions, Adamczewski and Bugeaud make the same conjecture \cite[Hypothesis~2]{adamczewski07_dynam_dioph_approx}. Even for this subclass proving this conjecture is considered out of reach, and no recent progress has been made \cite{adamczewski}.

However, as we will now show, the property in~\eqref{eq:hyp1} does not hold for general expanding \pam.\footnote{The property probably holds for $\beta$-expansions, as conjectured by Adamczewski and Bugeaud, as well as for some other classes of \pam, \eg complete maps; however it is too strong for general expanding \pam.}
\begin{example}
  Define the \pam $f$ as:
  \begin{align*}
  f(x)=
  \begin{cases}
    \frac 4 3 x  & \text{ if }0\le x < \frac 1 2,\\
    \frac 4 3 x - \frac 1 3 & \text{ if }\frac 1 2 \le x < 1.
  \end{cases}
  \end{align*}
  \begin{center}
      \mfbox{  
        \begin{minipage}{0.4\textwidth}
          \includegraphics[width=\textwidth]{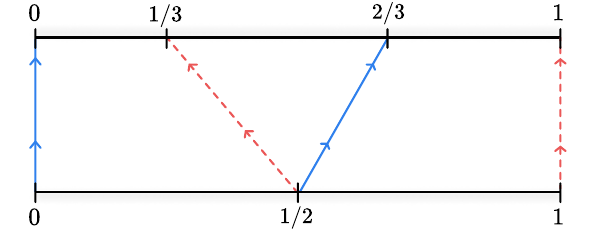}
        \end{minipage}
      }
    \end{center}    

  This is an expanding map as both slopes are $4/3>1$. Consider the orbit $\orbit(f,1/5)$. It is clear that it cannot be dense in all of $\unit$, because the first affine map, $4/3 x$, is monotone increasing, and the infimum of the image of the second affine map is:
  \begin{align*}
    \frac 4 3 \cdot \frac 1 2 - \frac 1 3 = \frac 1 3. 
  \end{align*}
  Therefore the orbit~$\orbit(f,1/5)$ cannot be dense in $\unit$, because \eg the interval $[0,1/5)$ does not intersect it. So it suffices to show that the orbit is infinite, in order to falsify claim \eqref{eq:hyp1}. We will show this by observing that $f$ decreases the $3$-adic valuation by at least one.

  Let $p$ be a prime number, and recall that the \textbf{$p$-adic valuation} of an integer $a$, denoted $v_p(a)$, is defined to be the largest positive integer $n$ such that
  \begin{align*}
    a=p^na',
  \end{align*}
  for some $a'\in \intg$. Then one extends this function to the rationals by defining $v_p(p/q)$ as $v_p(p)-v_p(q)$. We claim that for all $n\in\nat$
  \begin{align*}
    v_3\left(f^n\left(\frac 1 5\right)\right)>v_3\left(f^{n+1}\left(\frac 1 5\right)\right). 
  \end{align*}
  This claim plainly implies that the orbit $\orbit(f,1/5)$ is infinite.

  To prove the claim, observe that when the first affine map, $4/3\ x$, is applied, the $3$-adic valuation $v_3$ always decreases by $1$. As for the second affine map, let $p/q\in\rat$ be such that $q\equiv 0\ (\mathrm{mod}\ 3)$ and $p\not\equiv 0\ (\mathrm{mod}\ 3)$. Then when the second affine map is applied we have:
  \begin{align*}
    \frac 4 3 \cdot\frac p q - \frac 1 3 = \frac {4p-q} {3q}.
  \end{align*}
  Since $4p-q\not\equiv 0\ (\mathrm{mod}\ 3)$, we see that when the second affine map is applied to an irreducible rational whose denominator is divisible by $3$, then the $3$-adic valuation, $v_3$, decreases by~$1$. Finally, since every rational in the orbit, with the exception of $1/5$,  has this property, the claim follows. \qed
\end{example}

To complete this section, we prove a weaker property of expanding \pam, but along the lines of \eqref{eq:hyp1}. Let $X$ be a subset of $\rel$. We say that $p$ is an \textbf{accumulation point} of $X$ if every open neighbourhood of $p$ also contains some $p'\in X$, $p\ne p'$.
\begin{theorem}
  Let $f$ be an expanding \pam. For all $x_0\in\unit$, we have that either $\orbit(f,x_0)$ is finite, or it has infinitely many accumulation points. 
\end{theorem}
\begin{proof}
  Let $x_0\in\unit$ be such that $\orbit(f,x_0)$ is infinite, and towards a contradiction, assume that it has finitely many accumulation points: $p_1,\ldots,p_k$. Define the set $A$ as follows.\footnote{We are doing this refinement in order to take care of the cases when one of $p_1,\ldots, p_k$ is also an endpoint of an interval in the definition of $f$.} Put the symbol $p_i^{-}$ in $A$, $1\le i\le k$,  if and only if for every $\epsilon>0$ the orbit $\orbit(f,x_0)$ intersects the interval $(p_i-\epsilon,p_i)$. Similarly, put $p_i^{+}$ in the set $A$ if and only if for all $\epsilon>0$, the orbit intersects $(p_i,p_i+\epsilon)$.

  Now, for all $i\in\set{1,\ldots, k}$ and $s\in\set{-,+}$, such that $p_i^s\in A$, it is clear that there exists some $p_j^{t}\in A$ such that
  \begin{align*}
    \lim_{x\to p_i^s}f(x)=p_j^{t},
  \end{align*}
  since $p_1,\ldots, p_k$ are the only accumulation points of the orbit and $f$~is~a~\pam. Hence there is a graph structure with vertices $A$, and in fact this graph is a cycle. Consequently, there exists some $\ell\in\nat$ and $s\in\set{-,+}$ such that
  \begin{align}
    \label{eq:lim}
    \lim_{x\to p_1^s}f^{\ell}(x)=p_1^s. 
  \end{align}
  We argue that this is impossible. Indeed, it follows from \eqref{eq:lim} that all but finitely many elements of $\orbit(f^\ell,x_0)$ belong to a single interval in the definition of the \pam $f^\ell$. Therefore, after some threshold, the same affine map $g(x):=ax+b$ is being applied. But applying this affine map $n$ times is the same as applying
  \begin{align*}
    g^n(x)=a^nx+\frac {a^n-1} {a-1} b,
  \end{align*}
  as we saw in the proof of \Cref{lem:easy}. The orbit under $g$ cannot be Cauchy when $a>1$, yet in our case the latter holds, since $f$ is expanding and therefore so is $f^\ell$. 
\end{proof}

We do not know whether the methods in this paper can be used to prove stronger results. Experimenting with general \pam with two intervals, or bijective \pam with more than two intervals, one quickly realises that they produce orbits disimilar to the orbits that can be produced by the maps in this paper. For example, they can produce orbits that seem to be dense in only some sub-interval, which is not feasible with injective maps. 


\bibliographystyle{IEEEtran}
\bibliography{IEEEabrv,bibliography}
\end{document}